\documentclass[reqno]{amsart}
\usepackage{mathrsfs,amssymb,mathrsfs, multirow,setspace}
\usepackage{amsmath}
\usepackage[a4paper, total={7in, 9in}]{geometry}
\usepackage{enumerate}
\theoremstyle{plain}
\usepackage{graphicx}
\usepackage{mathtools}
\usepackage[utf8]{inputenc}

\newtheorem{theorem}{Theorem}[section]

\newtheorem{corollary}[theorem]{Corollary}

\theoremstyle{definition}

\theoremstyle{remark}

\input xy
\xyoption{all}

\begin{document}
\title[Laplacian Spectrum of Super Graphs defined on Certain Non-abelian Groups]{Laplacian Spectrum of Super Graphs defined on Certain Non-abelian Groups}
\author[Varun J Kaushik, Ekta, Parveen, Jitender Kumar]{Varun J Kaushik, Ekta, Parveen, Jitender Kumar$^{*}$}
\address{$\text{}^1$Department of Mathematics, Birla Institute of Technology and Science Pilani, Pilani, India}
\address{$\text{}^2$Department of Mathematics, Indian Institute of Technology Guwahati, India}

\email{f20190681@pilani.bits-pilani.ac.in, ektasangwan0@gmail.com, p.parveenkumar144@gmail.com, \newline jitenderarora09@gmail.com}
\begin{abstract}

Given a graph $A$ on a group $G$ and an equivalence relation $B$ on  $G$, the $B$ super$A$ graph, whose vertex set is $G$ and two vertices $g$, $h$ are adjacent if and only if there exist $g^{\prime} \in[g]$ and $h^{\prime} \in[h]$ such that $g^{\prime}$ and $h^{\prime}$ are adjacent in $A$. Recently, Dalal \emph{et al.}  (Spectrum of super commuting graphs of some finite groups, \textit{Computational and Applied Mathematics}, 43(6):348, 2024) obtain the Laplacian spectrum of supercommuting graphs of certain non-abelian groups including the dihedral group and the generalized quaternion group. In this paper, we continue the study of Laplacian spectrum of certian $B$ super$A$ graphs. We obtain the Laplacian spectrum of conjugacy superenhanced power graphs of certain non-abelian groups, namely: dihedral group, generalized quaternion group and semidihedral group. Moreover to enhance the work of Dalal \emph{et al}, we obtain the Laplacian spectrum of conjugacy supercommuting graph of semidihedral group. We prove that graphs considered in this paper are $L$-integral.


\end{abstract}



\maketitle

 \section{Introduction}

There are several types of graphs associated with group structures in the literature, such as power graphs \cite{a.MKsen2009}, enhanced power graphs \cite{a.Bera2017}, commuting graphs \cite{a.segev2001commuting}, cyclic graphs \cite{abdollahi2009noncyclic}, and Cayley graphs \cite{a.kelarev2002undirected}. The study of these algebraic graphs is of great importance due to their wide range of applications, including connections to automata theory, see \cite{kelarev2004labelled,a.kelarev2009cayley} and the books \cite{kelarev2002ring,kelarev2003graph}.

The undirected \emph{power graph} $\mathcal{P}(G)$ of a group $G$ is a graph whose vertex set is $G$ and two distinct vertices $x$, $y$ are adjacent if either $x = y^m$ or $y = x^n$ for some $m, n \in \mathbb{N}$.  Ali {\emph{et al.} \cite{ali2020powergraph} studied metric dimension and resolving polynomials of power graphs of certain non-abelian groups. Further, Chattopadhyay \emph{et al.} \cite{chattopadhyay2015Laplacian} studied the Laplacian spectrum of the power graph for the cyclic group $\mathbb{Z}_n$ and the dihedral group $D_{2n}.$ Then Mehranian {\emph{et al.}
\cite{Mehranian2017spectra}, explored the Laplacian spectrum of the power graph for the cyclic group $\mathbb{Z}_n, D_{2n}$, elementary abelian groups of prime power order and the Mathieu group $M_{11}.$ 
Additionally, Chattopadhyay \emph{et al.} \cite{chattopadhyay2018spectralradius} derive both the lower and the upper bounds for the spectral radius of the power graph associated with the cyclic group $\mathbb{Z}_n$. They also provide a characterization of the graphs that achieve these extremal bounds. Furthermore, the authors calculate the spectra of the power graphs corresponding to the dihedral group $D_{2n}$ and the generalized quaternion group $Q_{4n}.$ For more results on power graphs, the readers are referred to the survey paper \cite{ajay2021survey} and references therein.

The \textit{commuting graph} $Com(G)$ of a group $G$ is defined with vertex set $\Omega \subseteq G$ where two distinct vertices $x$ and $y$ are adjacent if $xy = yx$. When $\Omega = G$, the power graph $\mathcal{P}(G)$ is a spanning subgraph of the commuting graph $Com(G)$. Dutta \emph{et al.} \cite{dutta2018noncommuting} obtained the Laplacian spectrum of non-commuting graphs for certain non-abelian groups. Moreover, Subarsha \cite{banerarjee2023spectra} calculate the distance Laplacian spectra of commuting graph for certain non-abelian groups including dihedral group and generalized quaternion group. Aalipour \emph{et al.} \cite{a.Aalipour2017} characterized finite groups $G$ for which $\mathcal{P}(G)$ coincides with $Com(G)$ and introduced the \textit{enhanced power graph}. enhanced power graph $\mathcal{P}_E(G)$ is an undirected and simple graph whose vertices are elements of $G$ and two distinct vertices $x$ and $y$ are adjacent if they belong to the same cyclic subgroup $\langle z \rangle$ for some $z \in G$. They also characterized finite groups for which an arbitrary pair of the these graphs (power graph, enhanced power graph, commuting graph) coincide.

Bera \emph{et al.} \cite{a.Bera2017} characterized abelian groups and non-abelian $p$-groups with dominatable enhanced power graphs. Panda \emph{et al.} \cite{a.Panda-enhanced} studied graph-theoretic properties, such as minimum degree, independence number, matching number, strong metric dimension, and perfectness of enhanced power graphs over finite abelian and non-abelian groups, including dihedral and generalized quaternion group, as well as the group $U_{6n} = \langle a, b : a^{2n} = b^3 = e, \ ba = ab^{-1} \rangle$. Dalal \emph{et al.} \cite{a.dalal2021enhanced} investigated the properties of enhanced power graphs over generalized quaternion groups and the group $V_{8n} = \langle a, b : a^{2n} = b^4 = e, \ ba = a^{-1}b^{-1}, \ b^{-1}a = a^{-1}b \rangle$. Recently, Parveen \emph{et al.} \cite{parveen2024enhanced} explored the Laplacian spectrum of the enhanced power graph of semidihedral group, dihedral group, and generalized quaternion group, providing valuable insights into their structural properties. For a comprehensive overview of results and open questions on enhanced power graphs of groups, one can refer to the survey paper \cite{a.masurvey2022} and references therein. For a finite group $G$, there is a hierarchy containing the power graph, the enhanced power graph, and the commuting graph: one is a spanning subgraph of the next. Arunkumar \emph{et al.} \cite{arunkumar2022super}  explored the hierarchy formed by power, enhanced power, and commuting graphs, as well as modifications based on natural equivalence relations as follows.

Given a graph $A$ on a group $G$ and an equivalence relation $B$ on $G$, then the $B$ super$A$ graph, whose vertex set is $G$ and vertices $g$, $h$ are adjacent if and only if there exist $g^{\prime} \in[g]$ and $h^{\prime} \in[h]$ such that $g^{\prime}$ and $h^{\prime}$ are are adjacent in $A$. For example, let $G$ be a finite group and let $[x]$ be the conjugacy class of the element $x$ in $G$. The \emph{conjugacy superenhanced power graph} $CSEP(G)$ of the group $G$, is a simple undirected graph whose vertex set is $G$ and two vertices $x$, $y$ are adjacent if there exist $x^{\prime} \in[x]$ and $y^{\prime} \in[y]$ such that $x^{\prime}$ and $y^{\prime}$ are are adjacent in the enhanced power graph $\mathcal{P}_E(G)$. 

In \cite{arunkumar2022super}, the authors consider $B$ super$A$ graph, where $B$ is an equivalence relation, \emph{viz.} equality, conjugacy and same order, and $A \in \{\mathcal{P}(G), \mathcal{P}_E(G), Com(G)\}$. The resulting nine graphs (out of which eight are distinct in general) form a two-dimensional hierarchy. The authors of \cite{arunkumar2022super} identified conditions under which pair of these graphs coincide. This leads to the characterization of special group classes and the analysis of graph-theoretic properties such as completeness, dominant vertices, and clique numbers. Then Arunkumar \emph{et al.} \cite{arunkumar2024super} expressed certain supergraphs as graph compositions. The Wiener index of equality supercommuting and conjugacy supercommuting graphs for dihedral group and generalized quaternion group are determined in \cite{arunkumar2024super}. Moreover, Arunkumar \emph{et al.} \cite{arunkumar2024main} obtained the spectrum of equality supercommuting and conjugacy supercommuting graphs for dihedral group and generalized quaternion group, and proved that these graphs are not integral. Additionally, the Laplacian spectrum of conjugacy supercommuting graphs and order supercommuting graphs of dihedral and generalized quaternion group has been computed by Dalal \emph{et al.} \cite{dalal2024spectrum}. Motivated with the above work on $B$ super$A$ graphs, in this paper our aim is to study the Laplacian spectrum of conjugacy superenhanced power graphs of certain non-abelian groups, namely: dihedral group, generalized quaternion group and the semidihedral group. Moreover, to strengthen the work of Dalal \emph{et al.} \cite{dalal2024spectrum}, in this paper we obtain the Laplacian spectrum conjugacy supercommuting graphs for the semidihedral group.

\section{Preliminaries}
In this section, we recall necessary definitions, results and notations of graph theory from \cite{b.West}. A graph $\Gamma$ is an ordered pair $ \Gamma = (V, E)$, where $V = V(\Gamma)$ denotes the set of vertices and $E = E(\Gamma)$ denotes the set of edges in $\Gamma$. Furthermore, the \emph{order} of a graph $\Gamma$ is defined as the number of vertices in $\Gamma$. We say that two distinct vertices $a$ and $b$ are $\mathit{adjacent}$, denoted by $a \sim b$, if there is an edge connecting $a$ and $b$.  The \emph{neighbourhood} $N(x)$ of a vertex $x$ is the collection of all vertices which are adjacent to $x$ in the graph $ \Gamma $. Additionally, we denote $N[x] = N(x) \cup \{x\}$. We are considering simple graphs, i.e. undirected graphs with no loops or repeated edges. A \emph{subgraph} of a graph $\Gamma$ is defined as a graph $\Gamma'$ for which the vertex set $V(\Gamma') \subseteq V(\Gamma)$ and $E(\Gamma') \subseteq E(\Gamma)$. The subgraph $\Gamma(X)$ of a graph $\Gamma$, induced by a set $X$, consists of the vertex set $X$ where two vertices are connected by an edge if and only if they are adjacent in $\Gamma$. A graph $\Gamma$ is said to be a  $complete$ graph if every pair of distinct vertices are adjacent. We denote $K_n$ by the complete graph on $n$ vertices. The $complement \ \overline{\Gamma}$ of a simple graph $\Gamma$ is defined as a simple graph with the same vertex set $V(\Gamma)$. An edge $(u,v)\in E(\overline{\Gamma})$ if and only if the edge  $(u,v)\notin E(\Gamma)$. A graph that contains no cycles is referred to as $acyclic$. A $tree$ is defined as a connected acyclic graph. A $spanning \ subgraph$ of a graph $\Gamma$ is a subgraph that includes all the vertices of $\Gamma$, specifically the set $V(\Gamma)$. A $spanning \ tree$ is a spanning subgraph that meets the criteria of being a tree.
 The union of two graphs, $G_1$ and $G_2$, denoted as $G_1 \cup G_2$, is obtained by merging their vertex sets and edge sets: the vertex set is $V(G_1) \cup V(G_2)$ and the edge set is $E(G_1) \cup E(G_2)$. On the other hand, the join of two graphs $G_1$ and $G_2$, denoted as $G_1 \vee G_2$, is derived from this union by adding edges that connect every vertex in $G_1$  to every vertex in $G_2$.Consider $H$ be a graph with vertex set $V(H)= \{ 1, 2,\ldots, k \}$, and let $\Gamma_1,\ldots,\Gamma_k$ be the collection of graphs, each with the vertex set $V(\Gamma_i)=\{v^1_{i},\ldots,{v}^{n_i}_i\}$ for $1 \leq i \leq k$. Then the generalized composition of these graphs, denoted as $G=H[\Gamma_1, \ldots, \Gamma_k]$, has a vertex set given by the disjoint union $V(\Gamma_1) \cup \cdots \cup V(\Gamma_k)$.
Two vertices $v^p_i$  and  $v^q_j$ in $G$  are adjacent if one of the following conditions holds:
\\ (i) $i = j$ and $v^p_i$ and $v^q_i$ are adjacent in $ \Gamma_i$.
\\(ii) $ i \neq j $ and there is an edge between vertices $i$ and $j$ in $H$.

Let $\Gamma$ be a finite simple undirected graph with vertex set $V(\Gamma) = \{v_1, v_2, \ldots, v_n\}$. Then the \emph{adjacency matrix} $A(\Gamma)$ is an $n\times n$ matrix where the entry $(i, j)$ is $1$ if vertices $v_i$ and $v_j$ are adjacent, and $0$ otherwise. We denote the degree of vertex $v_i$ by $d_i$, leading to the diagonal matrix $D(\Gamma) = {\rm diag}(d_1, d_2, \ldots, d_n)$. The \emph{Laplacian matrix} $L(\Gamma)$ of $\Gamma$ is the matrix $D(\Gamma) - A(\Gamma)$. The matrix $L(\Gamma)$ is symmetric and positive semidefinite, it implies that its eigenvalues are real and non-negative. Moreover, the sum of the entries in each row (column) of $L(\Gamma)$ is equal to zero. The \emph{characteristic polynomial} of the Laplacian matrix $L(\Gamma)$ is denoted as $\Phi(L(\Gamma), x)$. The eigenvalues of $L(\Gamma)$, known as the \emph{Laplacian eigenvalues} of the graph $\Gamma$, are denoted by : $\lambda_1(\Gamma) \geq \lambda_2(\Gamma) \geq \cdots \geq \lambda_n(\Gamma) = 0$. Let us denote the distinct eigenvalues of $\Gamma$ by $\lambda_{n_1}(\Gamma) \geq \lambda_{n_2}(\Gamma) \geq \cdots \geq \lambda_{n_r}(\Gamma) = 0$, with multiplicities $m_1, m_2, \ldots, m_r$, respectively. The \emph{Laplacian spectrum} of $\Gamma$, that is, the spectrum  of $L(\Gamma)$ is denoted as $\displaystyle \begin{pmatrix}
\lambda_{n_1}(\Gamma) & \lambda_{n_2}(\Gamma) & \cdots& \lambda_{n_r}(\Gamma)\\
 m_1 & m_2 & \cdots & m_r\\
\end{pmatrix}$. Let $J_n$ denote the square matrix of order $n$ matrix where every entry is equal to $1$. Further, let $I_n$ denote the identity matrix of order $n$.

A \textit{maximal cyclic subgroup} of a group $G$ is a cyclic subgroup that is not a proper subgroup of any other cyclic subgroup of $G$. Let $(x)^G$ denote the \emph{conjugacy class} $\{gxg^{-1} : g \in G\}$ of $x\in G.$

For $n \geq 3$, the \emph{dihedral group} $D_{2n}$ is a group of order $2n$ defined as: 
\[
D_{2n} = \langle a, b : a^{n} = b^2 = e, \ ba = a^{-1}b \rangle.
\]
Note that the group $D_{2n}$ has one maximal cyclic subgroup $M = \langle a \rangle$ of order $n$ and $n$ maximal cyclic subgroups
$M_i =  \langle a^ib \rangle,$ where $1 \leq i \leq n$, of order $2.$

For $n \geq 2$, the \emph{generalized quaternion group} $Q_{4n}$ is a group of order $4n$ defined as:
\[
Q_{4n} = \langle a, b : a^{2n} = e, a^{n} = b^2, ba = a^{-1}b \rangle.
\]
Observe that the group
$Q_{4n}$ has one maximal cyclic subgroup $M = \langle a \rangle$ of order $2n$ and $n$ maximal cyclic subgroups
$M_i =  \langle a^ib \rangle,$ where $1 \leq i \leq n$, of order $4.$

For $n \geq 2$, the \emph{semidihedral group} $SD_{8n}$ is a group of order $8n$ defined by the generators and relations:

\[
SD_{8n} = \langle a, b : a^{4n} = e, \, b^2 = e, \, ba = a^{2n - 1}b \rangle.
\]
We have $$ ba^i = \left\{ \begin{array}{ll}
a^{4n -i}b & \mbox{if $i$ is even,}\\
a^{2n - i}b& \mbox{if $i$ is odd,}\end{array} \right.$$
so that every element of $SD_{8n} {\setminus} \langle a \rangle$ is of the form $a^ib$ for some $0 \leq i \leq 4n-1$. We denote the subgroups $P_i = \langle a^{2i}b \rangle = \{e, a^{2i}b\}$ and $ Q_j =  \langle a^{2j + 1}b \rangle = \{e, a^{2n}, a^{2j +1}b, a^{2n + 2j +1}b\} $. Then we have
\begin{equation}\label{SD_8n}
    SD_{8n} = \langle a \rangle \cup \left( \bigcup\limits_{ i=0}^{2n-1} P_i \right) \cup \left( \bigcup\limits_{ j=  0}^{n-1} Q_{j}\right),
\end{equation}
further, \[\ 
 Z(SD_{8n}) = \begin{cases}
    \{e,a^{2n}\},& \text{if $n$ is even} \\
    \{e,a^{2n},a^n,a^{3n}\},& \text{if $n$ is odd}.
    \end{cases}
\]
Note that the group $SD_{8n}$ has one maximal cyclic subgroup $M = \langle a \rangle$ of order $4n$, $2n$ maximal cyclic subgroups
$M_i =  \langle a^{2i}b \rangle = \{e, a^{2i}b \},$ where $1 \leq i \leq 2n$, of order $2$ and $n$ maximal cyclic subgroups $M_j =  \langle a^{2j+1}b \rangle = \{e, a^{2n}, a^{2j+1}b, a^{2n+2j+1}b \},$ where $1 \leq j \leq n$, of order $4$.

 In this paper, we obtain the Laplacian spectrum of conjugacy superenhanced power graphs of dihedral group, generalized quaternion group and semidihedral group. Also, we obtained the Laplacian of conjugacy supercommuting graphs of semidihedral group.
\section{Laplacian Spectrum of Conjugacy Superenhanced Power Graphs} 
In this section, we will obtain the structure of conjugacy superenhanced power graphs of the dihedral, dicyclic, and semidihedral groups and deduce their corresponding Laplacian spectra. The conjugacy superenhanced power graph ($CSEP$) of a group is analogous to that of the conjugacy supercommuting graphs, where the commuting graph is replaced by the enhanced power graph and conjugacy is the equivalence relation for these graphs as well.
\begin{theorem}
For the dihedral group $D_{2n}$, the structure of conjugacy superenhanced power graph of $D_{2n}$ is given by

    \[
CSEP(D_{2n})= 
\begin{dcases}
    K_1 \lor(K_{\frac{n-1}{2}}\cup K_1)[K_1,K_2,K_2,\ldots,K_2,K_n], \text{ if n is odd}   \\
   K_1 \lor(K_1\lor K_{\frac{n-2}{2}}\cup K_1\cup K_1)[K_1,K_1,K_2,K_2,\ldots,K_2,K_{\frac{n}{2}},K_{\frac{n}{2}}], \text{ if n is even}.   \\
\end{dcases}
\]
\end{theorem}
\vspace{0.5mm}
\begin{proof}
For the dihedral group $D_{2n}$, we have
    \[\ 
 Z(D_{2n}) = \begin{cases}
    \{e,a^{\frac{n}{2}}\},& \text{if~ $n$~is~ even} \\
    \{e\},& \text{if~ $n$~ is~ odd}
    .\end{cases}
    \]
The conjugacy classes of $D_{2n}$~ are~ given~ as~ follows:
\\ \noindent(i) \textit{If $n$ is odd:}
\begin{center}
$(e)^{D_{2n}}=\{e\},$ 

$(a^i)^{D_{2n}} =\{a^i,a^{n-i}\}, \text{ for } 1\leq i \leq n-1, i \neq \frac{n}{2},$

$(a^ib)^{D_{2n}} = \{a^{i}b :  0\leq i\leq n-1 \}, \text{ for } 1\leq i \leq n-1.$
\end{center} 
\noindent(ii) \textit{If $n$ is even:}
 \begin{center}
$(e)^{D_{2n}}=\{e\},~~   (a^{\frac{n}{2}})^{D_{2n}}= \{a^\frac{n}{2}\},$

$(a^i)^{D_{2n}} =\{a^i,a^{n-i}\} \text{ for } 1\leq i \leq n-1, i \neq \frac{n}{2},$

$(a^{2i-1}b)^{D_{2n}} = \{a^{2i+1}b : 1 \leq i \leq \frac{n}{2} \},$

$(a^{2i}b)^{D_{2n}} =\{a^{2i}b : 1 \leq i \leq \frac{n}{2} \}.$
\end{center} 
 We know that all the maximal cyclic subgroups of $D_{2n}$ are
$\langle a^{i}b \rangle = \{e, a^{i}b\},$ for $ 1 \leq i \leq n$ and $\langle a \rangle = \{a^i: 1 \leq i \leq n$\}.\\    
Now, we note that all the elements of the form $\{a^i\}$ are adjacent to each other in $\mathcal{P}_E(D_{2n})$. Furthermore, the identity element is adjacent to all other vertices of $\mathcal{P}_E(D_{2n})$. Observe that the elements of the classes $\{a^ib\}$ are adjacent to each other and the identity element. So we get the aforementioned graphs.
\end{proof}
\begin{theorem}\label{CSEP-odd-D_2n}  
If $n$ is odd, then the characteristic polynomial of the Laplacian matrix of ${CSEP}(D_{2n})$ is given by
\[\Phi(L({CSEP}(D_{2n})), x) = x(x-1)(x-(n+1))^{n-1}(x -2n)(x-n)^{n-2}.\]
\end{theorem}
\begin{proof}
The Laplacian matrix $L(CSEP(D_{2n}))$ is an $2n \times 2n$ matrix defined as follows, in which rows and columns are arranged in order by the vertices $e = a^{n}=b^{2}, a, a^2, \ldots, a^{n-1},$ and then $ab, a^2b,\ldots, a^{n}b=b$.   
\[
L(CSEP(D_{2n})) = \begin{pmatrix}
2n-1 & -1 & & \cdots & -1 & -1 & \cdots & -1 \\
-1 & & &  A & & & \mathcal{O} & \\
\vdots &  \\
-1 & &  & &  &  &  & \\
-1 &  &  & \mathcal{O}' & & \ & B & \\
\vdots & \\
-1 & & & & & & & & \\
\end{pmatrix},
\]
where \[A =\displaystyle \begin{pmatrix}
	n-1 & -1 & \cdots & -1\\
        -1 & n-1 & \cdots & -1\\
        \vdots & \vdots & \ddots & \vdots\\
        -1 & -1 & \cdots & n-1\\
	\end{pmatrix}\] 
and	 
        \[ B =\displaystyle \begin{pmatrix}
		n & -1 & \cdots & -1\\
        -1 & n & \cdots & -1\\
        \vdots & \vdots & \ddots & \vdots\\
        -1 & -1 & \cdots & n\\
	\end{pmatrix},
 \] \\
such that the orders of $A$ and $B$ are $n-1\times n-1$ and $n\times n$, respectively. Also, $\mathcal{O}$ is the zero matrix of order $(n-1)\times n$ and $\mathcal{O'}$ is the transpose of $\mathcal{O}$ matrix. It implies that 
$A = nI_{n-1} - J_{n-1}$ and
$ B = (n+1)I_{n} - J_{n}$.
 Then the characteristic polynomial of $L({CSEP}(D_{2n}))$ is
\[
\Phi(L(CSEP(D_{2n})), x) =
 \begin{vmatrix}
x-(2n-1) & 1 & \cdots & 1 & 1 & \cdots & 1 \\
1 & & xI_{n-1}-A & & & \mathcal{O} & \\
\vdots &  \\
1 & &  & & &  &  & \\
1 &  &  \mathcal{O}' & & \ & xI_{n}-B & \\
\vdots & \\
1 & & & & & & & & \\
\end{vmatrix}
.\]
Apply the row operation  $R_1 \rightarrow (x - 1)R_1 - R_2 - \cdots - R_{2n}$, and then expand by using the first row, we obtain
\[
\Phi(L(CSEP(D_{2n})), x) = x(x - 2n) \frac{1}{(x - 1)}
\det \begin{pmatrix}
(xI_{n-1} - A) & \mathcal{O} \\
\mathcal{O'} & (xI_{n} - B)
\end{pmatrix}
.\]

 Using the fact that if $M, Q$  are square matrices and the matrices $O$ and $ O'$ are zero matrices of appropriate sizes,  then $\left|\begin{array}{ll}M & O \\ O' & Q\end{array}\right|=|M|\cdot |Q|$ \cite[p.32]{b.spectra}. We get

\[
\Phi(L(CSEP(D_{2n})), x) = \frac{x(x - 2n)}{(x - 1)} \cdot {|xI_{n-1} - A|} \cdot {|xI_{n} - B|}.\
\]
Clearly, 
	$|xI_{n-1}-A| =(x-1)(x-n)^{n-2}$ 
 and	$|xI_{n}-B| =(x-1)(x-(n+1))^{n-1}$. This is because, if the eigenvalues of a matrix $A$ are $\{\sigma_1,\sigma_2,\ldots,\sigma_n\}$, then the eigenvalues of $\alpha I_n + A$ are given by $\{\alpha+\sigma_1,\alpha+\sigma_2,\ldots,\alpha+\sigma_n\}$. Hence, the characteristic polynomial of ${CSEP}(D_{2n})$ is $x(x-1)(x-(n+1))^{n-1}(x -2n)(x-n)^{n-2}$.
 
\end{proof}

\begin{theorem}
 \label{CSEP-D2n-odd}
If $n$ is odd, then the Laplacian spectrum of $CSEP(D_{2n})$ is given by
    \[\displaystyle \begin{pmatrix}
0 & 1 & n & n+1 & 2n\\
 1 & 1 & n-2 & n-1 & 1\\
\end{pmatrix}.\]
\end{theorem} 

By {\cite[Corollary 4.2]{mohar1991spectrum}}, we have the following corollary.
\begin{corollary}
If $n$ is odd, then the number of spanning trees of $CSEP(D_{2n})$ are $ n^{n-2} (n+1)^{n-1}.$
\end{corollary}
\begin{theorem}\label{CSEP-even-D_2n}  
If $n$ is even, then the characteristic polynomial of the Laplacian matrix of ${CSEP}(D_{2n})$ is given by:
\[\Phi(L({CSEP}(D_{2n})), x) =x(x-1)^{2}(x-(\frac{n}{2}+1))^{n-1}(x -2n)(x-n)^{n-2}.\]
\end{theorem}
\begin{proof}
The Laplacian matrix $L(CSEP(D_{2n}))$ is an $2n \times 2n$ matrix defined as follows, in which rows and columns are arranged in order by the vertices $e=a^n=b^{2}, a,a^2, \ldots, a^{n-1}$ followed by $ab,a^3b,\ldots, a^{n-1}b$ and finally $b,a^2b,\ldots,a^{n-2}b.$

\[L((CSEP(D_{2n}))  = \displaystyle \begin{pmatrix}
  2n-1 & -1  & \cdots  & -1 & -1 & \cdots & -1 & -1 & \cdots & -1 \\ 
 	-1  &  & A & & & \mathcal O & & & \mathcal O &  \\
 	\; \; \vdots & &  &  &  &  & \\ 
  -1 & & & & & & \\
 	-1  & &\mathcal O' & & & B & & &\mathcal O"&     \\ 

 	\; \; \vdots& & &  &  &  & \\   
   -1  & & & & & & \\
 	-1~ & &\mathcal O' &  & &\mathcal O" & & & B &   \\
  \;\; \vdots & & & & & & \\
  -1~ & & & & & & & \\
 	\end{pmatrix},
  \]
  where \[A =\displaystyle \begin{pmatrix}
	n-1 & -1 & \cdots & -1\\
        -1 & n-1 & \cdots & -1\\
        \vdots & \vdots & \ddots & \vdots\\
        -1 & -1 & \cdots & n-1\\
	\end{pmatrix}\] 
and	 
        \[ B =\displaystyle \begin{pmatrix}
		\frac{n}{2} & -1 & \cdots & -1\\
        -1 & \frac{n}{2} & \cdots & -1\\
        \vdots & \vdots & \ddots & \vdots\\
        -1 & -1 & \cdots & \frac{n}{2}\\
	\end{pmatrix},
 \] \\
such that the orders of $A$ and $B$ are $n-1\times n-1$ and $\frac{n}{2}\times \frac{n}{2}$, respectively. Also, $\mathcal{O}, \mathcal{O'}$ and $\mathcal{O''}$ are the zero matrices of appropriate sizes. Observe that $A = nI_{n-1} - J_{n-1}$ and $B = (\frac{n}{2}+1)I_{\frac{n}{2}} - J_{\frac{n}{2}}$.
 
 Then the characteristic polynomial of $L({CSEP}(D_{2n}))$ is
 
\[\Phi(L(CSEP(D_{2n})), x)  = 
 \displaystyle \begin{vmatrix}
  x-(2n-1)  & 1 & \cdots& 1 & 1 & \cdots& 1 & 1 & \cdots & 1 \\ 
 	1  &  & xI_{n-1}-A &  & & \mathcal O &  &  & \mathcal O& \\
 	\; \; \vdots &  &  &  &  &  & \\ 
  1 &  & & & & & \\
 1 & &\mathcal O' & & & xI_{2n}-B & & &\mathcal O"&     \\ 

 	\; \; \vdots& & &  &  &  &  \\ 
   1 & & & & & & \\
 	1~ & &\mathcal O' &  & &\mathcal O" & & & xI_{2n}-B &   \\
  \;\; \vdots & & & & & & \\
  1~ & & & & & & & \\
 	\end{vmatrix}.
  \] \\
After applying the following row operation
	\begin{itemize}
		\item $R_1 \rightarrow (x -1)R_1 - R_2  - \cdots - R_{2n}.$	
	\end{itemize}
we obtain 
\[
\Phi(L(CSEP(D_{2n})), x) = x(x - 2n) \frac{1}{(x - 1)}
\det \begin{pmatrix}
(xI_{n-1} - A) & \mathcal O & \mathcal O \\
\mathcal O' & (xI_{n} - B) & \mathcal O'' \\
\mathcal O' & \mathcal O'' & (xI_{n} - B)
\end{pmatrix}
\]
\[ = \frac{x(x - 2n)}{(x - 1)} \cdot {|xI_{n-1} - A|} \cdot {|xI_{\frac{n}{2}} - B|}^2\
.\]
Using a similar approach to the proof of Theorem \ref{CSEP-odd-D_2n},
we get the characteristic polynomial of ${CSEP}(D_{2n})$ for even $n$.
\end{proof}
\begin{corollary}\label{CSEP-D2n-even}
If $n$ is even, then the Laplacian spectrum of $CSEP(D_{2n})$ is given by
    \[\displaystyle \begin{pmatrix}
0 & 1 & \frac{n}{2}+1 & n & 2n\\
 1 & 2 & n-2 & n-2 & 1\\
\end{pmatrix}.\]
\end{corollary} 

By {\cite[Corollary 4.2]{mohar1991spectrum}}, we have the following corollary.
\begin{corollary}
If $n$ is even, then the number of spanning trees of $CSEP(D_{2n})$ are  $n^{n-2} (\frac{n}{2}+1)^{n-2}.$
\end{corollary}
\begin{theorem}
For the generalized quaternion group $Q_{4n}$, the structure of conjugacy superenhanced power graph of $Q_{4n}$ is given by
\[
CSEP(Q_{4n})= 
\begin{dcases}
    K_2 \lor(K_{n-1}\cup K_2)[K_1,K_2,K_2,\ldots,K_2,K_n,K_n], \text{ if n is odd}   \\
   K_2 \lor(K_{n-1}\cup K_1 \cup K_1)[K_1,K_2,K_2,\ldots,K_2,K_n,K_n], \text{ if n is even}. \\
\end{dcases}
\]    
\end{theorem}
\begin{proof}
 For the generalized quaternion group $Q_{4n}$, we have
 \[\ 
 Z(Q_{4n}) =
    \{e,a^{n}\}
    \]    
The conjugacy classes of the elements of generalized quaternion group $Q_{4n}$ are: 
 \begin{center}
$(e)^{Q_{4n}}=\{e\},~~   (a^{n})^{Q_{4n}}= \{a^n\},$

$(a^i)^{Q_{4n}} =\{a^i,a^{2n-i}\} \text{ for } 1\leq i \leq n-1, i \neq n,$

$ C_1= (a^{2i-1}b)^{Q_{4n}} = \{a^{2i-1}b : 1 \leq i\leq n \}, \text{ for } 1\leq i \leq n$

$C_2= (a^{2i}b)^{Q_{4n}} =\{a^{2i}b : 1 \leq i \leq n\}, \text{ for } 1 \leq i \leq n.$
\end{center}
We know that all the maximal cyclic subgroups of $Q_{4n}$ are $\langle a^ib \rangle = \{e, a^ib,a^n,a^{i+n}b\}$, for $1 \leq i \leq n$   and $\langle a\rangle = \{a^i: 1 \leq i \leq 2n\}.$ 
Note that all the elements of the form $\{a^i\}$ are adjacent to each other in $\mathcal{P}_E(Q_{4n})$. Furthermore, the identity element and the element $a^n$ are adjacent to all other vertices of $\mathcal{P}_E(Q_{4n})$. 
Now consider the cases

\noindent(i) \textit{ If $n$ is odd}.
    Notice that $ab\in C_1$ and $ab$ is adjacent to $a^{n+1}b$  in $\mathcal{P}_E(Q_{4n}).$ Also, $a^{n+1}b \in C_2.$ It implies that each element of $C_1$ is adjacent to every element of $C_2$ in $CSEP(Q_{4n})$.

\noindent(ii) \textit{ If $n$ is even}. In this case, $n+i$ remains even or odd according to $i$ being even or odd, respectively. It follows that no element of $C_1$ is adjacent to any element of $C_2$ in $CSEP(Q_{4n})$.

Hence, we get the above mentioned graphs.
\end{proof}

\begin{theorem}\label{CSEP-odd-Q_4n}  
If $n$ is odd, then the characteristic polynomial of the Laplacian matrix of ${CSEP}(Q_{4n})$ is given by
\[\Phi(L({CSEP}(Q_{4n})), x) = x(x-2)(x-(2n+2))^{2n-1}(x-4n)^2(x-2n)^{2n-3}.\]
\end{theorem}
\begin{proof}
The Laplacian matrix $L(CSEP(Q_{4n}))$ is an $4n \times 4n$ matrix defined as follows, in which rows and columns are arranged in order by the vertices $e = a^{2n}, a^{n}, a, a^2, \ldots, a^{n-1},a^{n+1}, \ldots, a^{2n-1}$ followed by $ab,a^2b,\ldots, a^{2n}b=b $.   
\[
L(CSEP(Q_{4n})) = \begin{pmatrix}
4n-1 & -1 & -1 & \cdots & -1 & -1 & \cdots & -1 \\
-1 & 4n-1 & -1 & \cdots & -1 & -1 & \cdots & -1 \\
-1 & -1 & &  A & & & \mathcal{O} & \\
\vdots & \vdots \\
-1 & -1 &  & &  &  &  & \\
-1 & -1 &  & \mathcal{O}' & & \ & B & \\
\vdots & \vdots \\
-1 & -1 & & & & & & & \\
\end{pmatrix}
,\]
where \[A =\displaystyle \begin{pmatrix}
	2n-1 & -1 & \cdots & -1\\
        -1 & 2n-1 & \cdots & -1\\
        \vdots & \vdots & \ddots & \vdots\\
        -1 & -1 & \cdots & 2n-1\\
	\end{pmatrix}\] 
and	 
        \[ B =\displaystyle \begin{pmatrix}
		2n+1 & -1 & \cdots & -1\\
        -1 & 2n+1 & \cdots & -1\\
        \vdots & \vdots & \ddots & \vdots\\
        -1 & -1 & \cdots & 2n+1\\
	\end{pmatrix},
 \] \\
such that the orders of $A$ and $B$ are $2n-2\times 2n-2$ and $2n\times2n$, respectively. Also, $\mathcal{O}$ and $\mathcal{O'}$ are the zero matrices of appropriate sizes. Note that $A = 2nI_{2n-2} - J_{2n-2}$ and $B = (2n+2)I_{2n} - J_{2n}$.
 
Then the characteristic polynomial of $L({CSEP}(Q_{4n}))$ is
\[
\Phi(L(CSEP(Q_{4n})), x) =
 \begin{vmatrix}
x-(4n-1) & 1 & 1 & \cdots & 1 & 1 & \cdots & 1 \\
1 & x-(4n-1) & 1 & \cdots & 1 & 1 & \cdots & 1 \\
1 & 1 & &  xI_{2n-2}-A & & & \mathcal{O} & \\
\vdots & \vdots \\
1 & 1 &  & &  &  &  & \\
1 & 1 &  & \mathcal{O}' & & \ & xI_{2n}-B & \\
\vdots & \vdots \\
1 & 1 & & & & & & & \\
\end{vmatrix}
.\]
After applying the following row operations consecutively
	\begin{itemize}
		\item $R_1 \rightarrow (x -1)R_1 - R_2  - \cdots - R_{4n}.$
		\item $R_2 \rightarrow (x -2)R_2 - R_3 - \cdots - R_{4n}.$	
	\end{itemize}
we obtain, 
\[
\Phi(L(CSEP(Q_{4n})), x) = x(x - 4n)^2 \frac{1}{(x - 2)}
\det \begin{pmatrix}
(xI_{2n-2} - A) & \mathcal{O} \\
\mathcal{O}' & (xI_{2n} - B)
\end{pmatrix}
\]
\[ = \frac{x{(x - 4n)}^2}{(x - 2)} \cdot {|xI_{2n-2} - A|} \cdot {|xI_{2n} - B|}\
.\]
Using a similar approach to the proof of Theorem \ref{CSEP-odd-D_2n},
we obtain  the spectrum of the graph as proposed.
\end{proof}
\begin{corollary} \label{CSEP-Q4n-odd}
If $n$ is odd, then the Laplacian spectrum for the $CSEP(Q_{4n})$ is given by
    \[\displaystyle \begin{pmatrix}
0 & 2 & 2n & 2n+2 & 4n\\
 1 & 1 & 2n-3 & 2n-1 & 2\\
\end{pmatrix}.\]
\end{corollary} 
By {\cite[Corollary 4.2]{mohar1991spectrum}}, we have the following corollary.

\begin{corollary}
If $n$ is odd, then the number of spanning trees of $CSEP(Q_{4n})$ are $2^{2n} n^{2n-2} (2n+2)^{2n-1}.$
\end{corollary}
\begin{theorem}\label{CSEP-even-Q_4n}  
If $n$ is even, then the characteristic polynomial of the Laplacian matrix of ${CSEP}(Q_{4n})$ is given by
\[\Phi(L({CSEP}(Q_{4n})), x) = x(x-2)^2(x-(n+2))^{2n-2}(x-4n)^2(x-2n)^{2n-3}.\]
\end{theorem}
\begin{proof}
The Laplacian matrix $L(CSEP(Q_{4n}))$ is an $4n \times 4n$ matrix defined as follows, in which the rows and columns are arranged in order by the vertices $e = a^{2n}, a^{n}, a, a^2, \ldots, a^{n-1},a^{n+1}, \ldots, a^{2n-1}$ followed by $ab,a^3b,\ldots, a^{2n-1}b $ and then finally $ a^2b, a^4b, \ldots, a^{2n}b=b.$   
\[
L(CSEP(Q_{4n})) = \begin{pmatrix}
4n-1 & -1 & -1 & \cdots & -1 & -1 & \cdots & -1 & -1 & \cdots~~~~ -1 \\
-1 & 4n-1 & -1 & \cdots & -1 & -1 & \cdots & -1 & -1 & \cdots~~~~  -1 \\
-1 & -1 & &  A & & & \mathcal{O} & & & \mathcal{O} \\
\vdots & \vdots \\
-1 & -1 &  & &  &  &  & \\
-1 & -1 &  & \mathcal{O}' & & \ & B & & & \mathcal{O}''  \\
\vdots & \vdots \\
-1 & -1 & & & & & & & \\
-1 & -1 & & & & & & & \\
\vdots & \vdots \\
-1 & -1 &  & \mathcal{O}' & & \ & \mathcal{O}''  & & & B \\
\end{pmatrix}
,\]
where \[A =\displaystyle \begin{pmatrix}
	2n-1 & -1 & \cdots & -1\\
        -1 & 2n-1 & \cdots & -1\\
        \vdots & \vdots & \ddots & \vdots\\
        -1 & -1 & \cdots & 2n-1\\
	\end{pmatrix}\] 
and	 
        \[ B =\displaystyle \begin{pmatrix}
		n+1 & -1 & \cdots & -1\\
        -1 & n+1 & \cdots & -1\\
        \vdots & \vdots & \ddots & \vdots\\
        -1 & -1 & \cdots & n+1\\
	\end{pmatrix},
 \] \\
such that the orders of $A$ and $B$ are $2n-2\times 2n-2$ and $n\times n$, respectively. Also, $\mathcal{O}, \mathcal{O'}$ and $\mathcal{O''}$  are the zero matrices of appropriate sizes. Observe that $A = 2nI_{2n-2} - J_{2n-2}$ and $B = (n+2)I_{n} - J_{n}$.
 
 Then the characteristic polynomial of $L({CSEP}(Q_{4n}))$ is
\[
\Phi(L(CSEP(Q_{4n})), x) =
\begin{vmatrix}
x-(4n-1) & 1 & 1 & \cdots & 1 & 1 & \cdots & 1 & 1 & \cdots~~~~ 1 \\
1 & x-(4n-1) & 1 & \cdots & 1 & 1 & \cdots & 1 & 1 & \cdots~~~~  1 \\
1 & 1 & & xI_{2n-2}-A & & & \mathcal{O} & & & \mathcal{O} \\
\vdots & \vdots \\
1 & 1 &  & &  &  &  & \\
1 & 1 &  & \mathcal{O}' & & \ & xI{n}-B & & & \mathcal{O}''  \\
\vdots & \vdots \\
1 & 1 & & & & & & & \\
1 & 1 & & & & & & & \\
\vdots & \vdots \\
1 & 1 &  & \mathcal{O}' & & \ & \mathcal{O}''  & & & xI{n}-B \\
\end{vmatrix}
.\]
After applying the following row operations consecutively
	\begin{itemize}
		\item $R_1 \rightarrow (x -1)R_1 - R_2  - \cdots - R_{4n}.$
		\item $R_2 \rightarrow (x -2)R_2 - R_3 - \cdots - R_{4n}.$	
	\end{itemize}
we obtain, 
\[
\Phi(L(CSEP(Q_{4n})), x) = x(x - 4n)^2 \frac{1}{(x - 2)}
\det \begin{pmatrix}
(xI_{2n-2} - A) & \mathcal{O} &  \mathcal{O} \\
\mathcal{O'} & (xI_{n} - B) & \mathcal{O''}\\
\mathcal{O'} & \mathcal{O''} & (xI_{n} - B) \\
\end{pmatrix}
\]
\[= \frac{x{(x - 4n)}^2}{(x - 2)} \cdot {|xI_{2n-2} - A|} \cdot {|xI_{n} - B|}^2\
.\]
By a similar approach used in the proof of Theorem \ref{CSEP-odd-D_2n}, we obtain  the spectrum of the graph as proposed.
\end{proof}

\begin{corollary}\label{CSEP-Q4n-even}
If $n$ is even, then the Laplacian spectrum for the $CSEP(Q_{4n})$ is given by
    \[\displaystyle \begin{pmatrix}
0 & 2 & n+2 & 2n & 4n\\
 1 & 2 & 2n-2 & 2n-3 & 2\\
\end{pmatrix}.\]
\end{corollary} 
By {\cite[Corollary 4.2]{mohar1991spectrum}}, we have the following corollary.

\begin{corollary}
If $n$ is even, then the number of spanning trees of $CSEP(Q_{4n})$ are $2^{2n+1} n^{2n-2} (n+2)^{2n-2}.$
\end{corollary}
\begin{theorem}
For the semidihedral group $SD_{8n}$, the structure of conjugacy superenhanced power graph of $SD_{8n}$ is given by

\[
CSEP(SD_{8n})= 
\begin{dcases}
     K_1\lor (K_1\lor(K_{2n-1}\cup K_1)\cup K_1)[K_1,K_1,K_2,K_2,\ldots, K_2,K_{2n},K_{2n}],~ \text{if $n$ is even} \\  K_1\lor (K_1\lor(K_{2n-1}\cup K_1)\cup K_1\cup K_1 )[K_1,K_1,K_2,K_2,\ldots, K_2,K_{2n},K_{n},K_{n}]
   , \text{ if $n$ is odd}.   \\ 
\end{dcases}
\]
\end{theorem}
\begin{proof}For the semidihedral group $SD_{8n}$, we have
\[\ 
 Z(SD_{8n}) = \begin{cases}
    \{e,a^{2n}\},& \text{if $n$ is even} \\
    \{e,a^{2n},a^n,a^{3n}\},& \text{if $n$ is odd}.
    \end{cases}
\]

Now, the conjugacy classes of semidihedral group $SD_{8n}$~ are~ given~ as~ follows:
\\ \noindent \textbf{Case 1:} \textit{If $n$ is odd.}

\begin{center}
$(e)^{SD_{8n}}=\{e\}$, $(a^{n})^{SD_{8n}}=\{a^{n}\}$, $(a^{2n})^{SD_{8n}}=\{a^{2n}\}$, $(a^{3n})^{SD_{8n}}=\{a^{3n}\}$,
\[ (a^{i})^{SD_{8n}}=
\begin{dcases}
\{a^i,a^{4n-i}\},~ i~ $\text{is even}$\\ 
\{a^i,a^{2n-i}\},~ i~ $\text{is odd}$, 
\end{dcases}
\]
$ C_{j}=\{a^{4k+j-1}b : 0\leq k\leq n-1 \},$ for $ 1 \leq j\leq 4.$ 
\end{center}

Note that, $\{a^{2n}\}$ is adjacent to $\langle a \rangle \cup 
  \{{a^{2i+1}b}: 1 \leq i \leq 2n-1\}$. For the conjugacy classes $\{a^{i}, \, a^{2n-i}\}, \text{where} ~ 1\leq i\leq 2n-1 $ and $\{a^{i},a^{4n-i}\}, \text{where} ~ 2 \leq i\leq 2n-2 $ for $i$ odd or even respectively, each element of these conjugacy classes are adjacent to every element of other conjugacy classes, this forms a complete graph of order ${4n-2}$.
Now $ab \in C_{2} $ and $ab$ is adjacent to $a^{2n+1}b$ in $P_{E}(SD_{8n}).$ Observe that $a^{2n+1}b \in C_{4}.$ It implies that each element of $C_{2}$ is adjacent to every element of $C_{4}.$

Now, we obtain the neighbours of each element of the conjugacy superenhanced power graph
\begin{enumerate} [\rm (i)]
\item $N[e]=SD_{8n}.$
\item $N[a^{2n}]= \langle a \rangle \cup 
  \{{a^{2i+1}b}: 0 \leq i \leq 2n-1\}.$
  \item $N[a^i]= \langle a \rangle ,~ \text{where}~ 1 \leq i \leq 4n-1, i \neq 2n.$
\item $N[a^{2i-1}b]=\{e,a^{2n},ab,a^{3}b, \ldots ,a^{4n-1}b\}, \text{ where}~ 1 \leq i \leq 2n.$
\item $N[a^{4i}b]=\{e,a^{4}b,a^{8}b,\ldots, a^{4n-4}b, a^{4n}b=b \}, \text{ where } 1 \leq i \leq n. $
\item $N[a^{4i-2}b]=\{e,a^{2}b,a^{6}b,\ldots, a^{4n-2}b\}, \text{ where } 1 \leq i \leq n. $
\end{enumerate}
\noindent  \textbf{Case 2:} \textit{If $n$ is even.} 
    \begin{center}
        $(e)^{SD_{8n}}=\{e\},(a^{2n})^{SD_{8n}}=\{a^{2n}\}\vspace{-0.5mm} $,       
\[ (a^{i})^{SD_{8n}}=
\begin{dcases}
\{a^i,a^{4n-i}\},~ i~ $\text{is even}$\\ 
\{a^i,a^{2n-i}\},~ i~ $\text{is odd},$ 
\end{dcases}
\]
$ D_1=\{a^{2k+1}b :  0\leq k\leq n-1 \},$\ $ D_2=\{a^{2k}b :  0\leq k\leq 2n-1 \}.$\

\end{center}

Now, for the conjugacy classes $\{a^{i}, \, a^{2n-i}\}, \text{where} ~ 1\leq i\leq 2n-1 $ and $\{a^{i},a^{4n-i}\}, \text{where} ~ 2 \leq i\leq 2n-2 $ for $i$ odd or even, respectively. Each element of these conjugacy classes are adjacent to every element of other conjugacy classes. It implies that this forms a complete graph of order ${4n-2}.$ In case of $n$ being even, $2n+i$ remains even or odd according to $i$ being even or odd, respectively.

Now, we obtain the neighbourhood of each element of the conjugacy superenhanced power graph:

\begin{enumerate} [\rm (i)]
\item $N[e]=SD_{8n}.$
\item $N[a^{2n}]= \langle a \rangle \cup 
  \{{a^{2i+1}b}: 0 \leq i \leq 2n-1\}.$
  \item $N[a^i]= \langle a \rangle ,~ \text{where}~ 1 \leq i \leq 4n-1, i \neq 2n.$
\item $N[a^{2i-1}b]=\{e,a^{2n},ab,a^{3}b, \ldots ,a^{4n-1}b\}, \text{ where}~ 1 \leq i \leq 2n.$
\item $N[a^{2i}b]=\{e,a^{2}b,a^{4}b,\ldots, a^{4n-2}b, a^{4n}b=b\}, \text{ where } 1 \leq i \leq 2n. $
\end{enumerate}

Looking at the adjacency of the elements in the enhanced power graph, we get the above mentioned graphs.
\end{proof}

\begin{theorem}\label{CSEP-SD8n}  
If $n$ is even, then the characteristic polynomial of the Laplacian matrix of ${CSEP}(SD_{8n})$ is given by
\[\Phi(L({CSEP}(SD_{8n})), x) = x(x-1)(x-2)(x-(2n+1))^{2n-1}(x-(2n+2))^{2n-2}(x -4n)^{4n-3}(x-6n)(x-8n).\] 
\end{theorem}
\begin{proof}
The Laplacian matrix $L({CSEP}(SD_{8n}))$ is an $8n \times 8n$ matrix defined as follows, in which rows and columns are arranged in order by the vertices $e = a^{4n}=b^{2},a^{2n}, a, a^2, \ldots, a^{2n-1},a^{2n+1},\ldots, a^{4n-1}$ and then $ab, a^3b, a^5b,  \ldots, a^{4n-1}b$ and finally, $b, a^2b, a^4b,\ldots, a^{4n-2}b$.


\[L(CSEP(SD_{8n}))  = \displaystyle \begin{pmatrix}
  8n-1 & -1 & -1 & \cdots  & -1 & -1 & \cdots &-1& -1\cdots -1 \\ 
  -1 & 6n - 1 & -1 & \cdots & -1 & -1 & \cdots&-1 & ~~0~ \cdots ~~~~~0 \\
 	-1  & -1 &  & A &   & & \mathcal O &    & \mathcal O& \\
 	\; \; \vdots & \; \;\vdots&    &  &  &  &  & \\ 
  -1 & -1 & & & & & & \\
 	-1  & -1 & &\mathcal O' & & & B & &\mathcal O"&     \\ 

 	\; \; \vdots& \; \; \vdots& & &  &  &  &  \\ 
   -1 & -1 & & & & & & \\
 	-1~  &~~~~~~~  0 & &\mathcal O' &  & &\mathcal O" &  & B' &   \\
  \;\; \vdots & \; \; \vdots & & & & & & \\
  -1~ &~~~~~~~ 0 & & & & & & & \\
 	\end{pmatrix},
  \]
  where
	\[A =\displaystyle \begin{pmatrix}
	4n-1 & -1 & \cdots & -1\\
        -1 & 4n-1 & \cdots & -1\\
        \vdots & \vdots & \ddots & \vdots \\
        -1 & -1 & \cdots & 4n-1\\
	\end{pmatrix},\] 
	 
        \[ B =\displaystyle \begin{pmatrix}
		2n+1 & -1 & \cdots & -1\\
        -1 & 2n+1 & \cdots & -1\\
        \vdots & \vdots & \ddots & \vdots\\
        -1 & -1 & \cdots & 2n+1\\
	\end{pmatrix}
 \]
and
 \[ B' =\displaystyle \begin{pmatrix}
		2n & -1 & \cdots & -1\\
        -1 & 2n & \cdots & -1\\
        \vdots & \vdots & \ddots & \vdots\\
        -1 & -1 & \cdots & 2n\\
	\end{pmatrix},
 \]
  such that the orders of $A$, $B$ and $B'$ are $4n-2\times 4n-2$, $2n\times 2n$ and $2n\times 2n$, respectively. It follows that $A= 4nI_{4n-2}-J_{4n-2}$, $B= (2n+2)I_{2n}-J_{2n}$ and $B'= (2n+1)I_{2n}-J_{2n}.$

Then the characteristic polynomial of $L({CSCom}(SD_{8n}))$ is
\[\Phi(L(CSEP(SD_{8n})), x)  = \displaystyle \begin{vmatrix}
  x-(8n-1) & 1 & 1 & \cdots  & 1 & 1 & \cdots & 1 & 1 ~~~~ \cdots~~~  1 \\ 
  1 & x-(6n - 1) & 1 & \cdots & 1 & 1 & \cdots& 1 & ~~0~ \cdots ~~~~~0 \\
 	1  & 1 &  & xI_{4n-2}-A &   & & \mathcal O &    & \mathcal O& \\
 	\; \; \vdots & \; \;\vdots&    &  &  &  &  & \\ 
  1 & 1 & & & & & & \\
 	1  & 1 & &\mathcal O' & & & xI_{2n}-B & &\mathcal O"&     \\ 

 	\; \; \vdots& \; \; \vdots& & &  &  &  &  \\ 
   1 & 1 & & & & & & \\
 	1~  &~~~~~~~  0 & &\mathcal O' &  & &\mathcal O" &  & xI_{2n}-B' &   \\
  \;\; \vdots & \; \; \vdots & & & & & & \\
  1~ &~~~~~~~ 0 & & & & & & & \\
 	\end{vmatrix}.
  \]
  
After applying the following row operations consecutively
	\begin{itemize}
		\item $R_1 \rightarrow (x -1)R_1 - R_2  - \cdots - R_{8n}.$
		\item $R_2 \rightarrow (x -2)R_2 - R_3 - \cdots - R_{6n}.$
	\end{itemize}
 we obtain, 
	
	\[\Phi(L(CSEP(SD_{8n})), x) = \frac{x(x -8n)(x-6n)}{(x-2)}  \displaystyle \begin{vmatrix}
	xI_{4n-2} -A & \mathcal O & \mathcal O\\
	\mathcal O' & xI_{2n} - B & \mathcal{O"}\\
 \mathcal{O'} & \mathcal{O"}  & xI_{2n} - B'
	\end{vmatrix} \]
 \vspace{3mm}
 \[
 = \frac{x(x -8n)(x-6n)}{(x-2)} |xI_{4n-2}- A|\cdot|xI_{2n} - B|\cdot|xI_{2n} - B'|.
 \]
Using a similar approach to the proof of Theorem \ref{CSEP-odd-D_2n}, we obtain  the spectrum of the graph as proposed.
\end{proof}
\begin{corollary}
If $n$ is even, then the Laplacian spectrum of $CSEP(SD_{8n})$ is given by 
\[\displaystyle \begin{pmatrix}
0 & 1& 2 & 2n+1 & 2n+2 & 4n & 6n & 8n\\
 1 & 1&1 & 2n-1 & 2n-1 & 4n-3 &1 &1\\
\end{pmatrix}.\]
\end{corollary}

By {\cite[Corollary 4.2]{mohar1991spectrum}}, we have the following corollary.

\begin{corollary}
If $n$ is even, then the number of spanning trees of $CSEP(SD_{8n})$ are $ 3.2^{8n-4} n^{4n-2} (2n+2)^{2n-1} \\(2n+1)^{2n-1}.$
\end{corollary}
\begin{theorem}
If $n$ is odd, then the characteristic polynomial of the Laplacian matrix of \(CSEP(SD_{8n})\) is given by
\[
\Phi(L(CSEP(SD_{8n})), x) = x(x - 8n)(x - 6n)(x - 2)(x - 4n)^{4n-3}(x - 1)^{2}(x-(n+1))^{2n-2}(x-(2n+2))^{2n-1}.
\]
\end{theorem}

\begin{proof}
   
The Laplacian matrix \( L(CSEP(SD_{8n})) \) is an \(8n \times 8n\) matrix defined as follows, where the rows and columns are arranged in order by the vertices $e = a^{4n}=b^{2},a^{2n}, a,a^2, \ldots, a^{2n-1},a^{2n+1},\ldots, a^{4n-1}$ and then $ab, a^3b,a^5b,  \ldots,a^{4n-1}b$ and, $b,a^4b,\ldots, a^{4n-4}b$ and finally $a^2b,a^6b,\ldots, a^{4n-2}b$.

 \[
L(CSCom(SD_{8n})) = \begin{pmatrix}
8n-1 & -1 & -1 & -1 & \cdots & -1 & -1 & \cdots & -1 \\
-1 & & & & A & & & \mathcal{O} & \\
-1 & & & & & & & & \\
-1 & & & & & & & & \\
\vdots &  \\
-1 & &  & &  &  &  & \\
-1 &  &  &  & \mathcal{O}' & & \ & B & \\
\vdots & \\
-1 & & & & & & & & \\
\end{pmatrix}
,\]

where

\[
A =
\begin{pmatrix}
6n-1 & -1 & -1 & -1 & \cdots & -1 & -1 & \cdots & -1 \\
-1 & & & & A' & & & \mathcal{O''} & \\
-1 & & & & & & & & \\
-1 & & & & & & & & \\
\vdots &  \\
-1 & &  & &  &  &  & \\
-1 &  &  &  & \mathcal{O'''} & \ & & B' & \\
\vdots & \\
-1 & & & & & & & & \\
\end{pmatrix}
\] and

\[
B =
\begin{pmatrix}
C & \mathcal{O}_{1} &\\
\mathcal{O}_{1} & C &\\
\end{pmatrix}
.\]

such that, $A' = 4nI_{4n-2} - J_{4n-2}$, $B' = (2n+2)I_{2n} - J_{2n}$ and $C = (n+1)I_{n} - J_{n}$

$\mathcal{O}_{1}, \mathcal{O}, \mathcal{O'}, \mathcal{O''}$ and $\mathcal{O}'''$ are the zero matrices of appropriate sizes.
Then the characteristic polynomial of  $L(CSEP(SD_{8n}))$ is given by

\[
\Phi(L(CSEP(SD_{8n})), x) =
 \begin{vmatrix}
x-(8n-1) & 1 & 1 & 1 & \cdots & 1 & 1 & \cdots & 1 \\
1 & & & & xI_{6n-1}-A & & & \mathcal{O} & \\
1 & & & & & & & & \\
1 & & & & & & & & \\
\vdots &  \\
1 & &  & &  &  &  & \\
1 &  &  &  & \mathcal{O}' & & \ & xI_{2n}-B & \\
\vdots & \\
1 & & & & & & & & \\
\end{vmatrix}
.\]

Apply the row operation \( R_1 \rightarrow (x - 1)R_1 - R_2 - \cdots - R_{8n} \) and then expand by using the first row, we obtain

\[
\Phi(L(CSEP(SD_{8n})), x) = x(x - 8n) \frac{1}{(x - 1)}
\det \begin{pmatrix}
(xI_{6n-1} - A) & O \\
O' & (xI_{2n} - B)
\end{pmatrix}
\]
\[= \frac{x(x - 8n)}{(x - 1)} \cdot {|xI_{6n-1} - A|} \cdot {|xI_{2n} - B|}\
.\]

Now, we have
\[
|xI_{6n-1} - A| =
\begin{vmatrix}
x-(6n-1) & 1 & 1 & 1 & \cdots & 1 & 1 & \cdots & 1 \\
1 & & & & xI_{4n-2}-A' & & & \mathcal{O''} & \\
1 & & & & & & & & \\
1 & & & & & & & & \\
\vdots &  \\
1 & &  & &  &  &  & \\
1 &  &  &  & \mathcal{O'''} & \ & & xI_{2n}-B' & \\
\vdots & \\
1 & & & & & & & & \\
\end{vmatrix}
.\]

Apply row operation \( R_1 \to (x - 2)R_1 - R_2 - R_3 - \cdots - R_{6n-1} \) 
and then expand by using the first row, we get
\[
|xI_{6n-1} - A| = \frac{(x - 1)(x - 6n)}{(x - 2)} |xI_{4n-2} - A'| \cdot |xI_{2n} - B'|.
\]
Now
\[
|xI_{4n-2} - A'|= |xI_{4n-2}-((4n)I_{4n-2}-J_{4n-2})|
,\]
\[
|xI_{2n} - B'|= |xI_{2n}-((2n+2)I_{2n}-J_{2n})|
,\]
\[
|xI_{2n} - B| = |xI_{n} - C| \cdot |xI_{n} - C|
,\]
\[
|xI_{n} - C|= |xI_{n}-((n+1)I_{n}-J_{n})|
.\]
From the computation, we obtain
\[
|xI_{4n-2} - A'| = (x - 4n)^{4n-3} (x - 2)
.\]
\[
|xI_{2n} - B'|= (x - 2) (x - (2n+2))^{2n-1}, 
\]
\[
|xI_{2n} - B| = |xI_{n} - C| \cdot |xI_{n} - C|, 
\]
\[
|xI_{n} - C|= (x - 1) (x - (n+1))^{n-1}, 
\]
It implies that
\[
|xI_{2n} - B| = (x - 1)^2 (x - (n+1))^{2n-2}
\]
It follows the result.
\end{proof}
\begin{corollary}
If $n$ is odd, then the Laplacian spectrum of $CSEP(SD_{8n})$ is given by 
\[\displaystyle \begin{pmatrix}
0 & 1& 2 & n+1 & 2n+2 & 4n & 6n & 8n\\
 1 & 2&1 & 2n-2 & 2n-1 & 4n-3 &1 &1\\
\end{pmatrix}.\]
\end{corollary}

By {\cite[Corollary 4.2]{mohar1991spectrum}}, we have the following corollary.

\begin{corollary}
If $n$ is odd, then the number of spanning trees of $CSEP(SD_{8n})$ are $ 3.2^{8n-4} n^{4n-2} (2n+2)^{2n-1} \\(n+1)^{2n-2}.$
\end{corollary}
 \section{Laplacian Spectrum of Conjugacy Supercommuting Graphs}
Arunkumar $et \ al.$ \cite{arunkumar2022super}  explained the concept of super graphs for a graph type A and equivalence relation B on any given group structure. Now, we shall derive the supercommuting graphs for the semidihedral group $SD_{8n}$.

\begin{theorem}\label{SD8n-CSCommGraphBuilding}
For the semidihedral group $SD_{8n}$, the structure of conjugacy supercommuting graph of $SD_{8n}$ is given by
\[CSCom(SD_{8n})= 
\begin{dcases}
K_4 \lor(K_{2n-2}\cup K_4) [K_1,K_1,K_1,K_1,K_2,K_2,\ldots,K_2,K_{n},K_{n},K_{n},K_{n}], \text{if n is odd}  \\
K_2 \lor(K_{2n-1}\cup K_1 \cup K_1)[K_2,K_2,K_2,K_2,\ldots,K_2,K_{2n},K_{2n}], \text{if n is even}  \\
\end{dcases}
\]
\end{theorem}
\begin{proof}
We have seen the conjugacy classes of semidihedral group $SD_{8n}$ in Theorem $3.11.$
\\ \noindent\textbf{Case 1:} \emph{When $n$ is odd.} Note that $b \in C_{1}$ and $b$ is adjacent to $\{e,a^{n},a^{2n},a^{3n},a^{n}b,a^{2n}b,a^{3n}b\}$ in $Com(SD_{8n})$. Observe that $a^{2n}b \in C_{3}.$ If $a^{n}b \in C_{2},$ then $a^{3n}b \in C_{4}.$ Otherwise, $a^{n}b \in C_{4} $ and $a^{3n}b \in C_{2}.$ 
It implies that each element of $C_{1}$ is adjacent to every element of $C_{2}, C_{3}~  \text{and}~ C_{4}, $ respectively. 
 Similarly, each element of $C_{i}$ is adjacent to every element of $C_{j}$ for $ 1\leq i,j\leq 4.$
Now, $\{a^{i}b\}$ commutes with $\{a^k\}$ if and only if $k \in \{0,n,2n,3n\}.$ Also $\{a^k\}
\notin C_{j}$ for $1 \leq j\leq 4.$ It implies that $\{a^{i}b\}$ is not adjacent to $\{a^k\}$ in $CSCom(SD_{8n}).$ 
 In the conjugacy supercommuting graph of the semidihedral group, the neighbourhood of each element is given as:



\begin{enumerate}
    \item[(i)] $N[e]= N[a^n]= N[a^{2n}]= N[a^{3n}]= SD_{8n}. $
         \item[(ii)] $N[x]= \{a^ib:0 \leq i \leq 4n-1\}\cup \{e, a^n, a^{2n}, a^{3n}\}$ if and only if $x\in \{a^jb:0 \leq j \leq 4n-1\}$.
 \item[(iii)] $N[x]=\langle a \rangle$ if and only if $x \in \langle a \rangle$\textbackslash
 $\{e, a^n,a^{2n},a^{3n}\}.$ 
    \end{enumerate}

\noindent\textbf{Case 2:} \emph{When $n$ is even.} In case of $n$ being even, $2n+i$ remains even or odd, respectively. It implies that no element of $D_1$ is adjacent to any element of $D_2$ in ${CSCom}(SD_{8n}).$ In the conjugacy supercommuting graph of the semidihedral group, the neighbourhood of each element is given as:
\begin{enumerate}
\item[(i)] $N[e]= N[a^{2n}] = SD_{8n}. $
\item[(ii)] $N[x]=\{e,a^{2n}\}\cup\{ a^{2i}b: 0 \leq i \leq 2n-1\}$ if and only if $x \in \{a^{2i}b: 0 \leq i \leq 2n-1\}$.

\item[(iii)] $N[x]=\{e,a^{2n}\}\cup \{a^{2j+1}b:0 \leq j \leq 2n-1\}$ if and only if $x \in
 \{a^{2j+1}b: 0 \leq j \leq 2n-1\}.$

 \item[(iv)] $N[x]=\langle a \rangle$ if and only if $x \in \langle a \rangle$\textbackslash
$\{e, a^{2n}\}.$
\end{enumerate}


    \vspace{0.5mm}

Hence, we get the aforementioned graphs.
\end{proof}

 Another representation of Theorem \ref{SD8n-CSCommGraphBuilding} is given in {\cite [Theorem 2.5]{das2024super}}.

\begin{theorem} \label{SD8n-CSCom-n odd}
If $n$ is odd, then the characteristic polynomial of the Laplacian matrix of ${CSCom}(SD_{8n})$ is given by 
\[\Phi(L({CSCom}(SD_{8n})), x) = x(x-4)(x-4n)^{4n-5}(x- (4n+4))^{4n-1}(x-8n)^{4}.\]
\end{theorem}
\begin{proof}
    The Laplacian matrix $L({CSCom}(SD_{8n}))$ is an $8n \times 8n$ matrix given below, in which the rows and columns are indexed in order by the vertices $e = a^{4n}=b^{2},a^{n},a^{2n},a^{3n}, a,a^{2},
    \ldots, a^{n-1},a^{n+1},\ldots ,a^{2n-1}, a^{2n+1},\ldots, a^{3n-1}, a^{3n+1}, \ldots, a^{4n-1}$ and then $b, a^{2}b,a^{4}b,  \ldots,a^{4n-2}b$ and finally we have,
    $ab, a^{3}b,a^{5}b,  \ldots,a^{4n-1}b$.

 \[
L(CSCom(SD_{8n})) = \begin{pmatrix}
8n-1 & -1 & -1 & -1 & \cdots & -1 & -1 & \cdots & -1 \\
-1 & 8n-1 & -1 & -1 & \cdots & -1 & -1 & \cdots & -1 \\
-1 & -1 & 8n-1 & -1 & \cdots & -1 & -1 & \cdots & -1 \\
-1 & -1 & -1 & 8n-1 & \cdots & -1 & -1 & \cdots & -1 \\
-1 & -1 & -1 & -1 & A & & & \mathcal{O} & \\
\vdots & \vdots & \vdots & \vdots &  \\
-1 & -1 & -1 & -1 & \mathcal{O}' & & \ & B & \\
\end{pmatrix},
\] 
 where
	\[A =\displaystyle \begin{pmatrix}
	4n-1 & -1 & \cdots & -1\\
        -1 & 4n-1 & \cdots & -1\\
        \vdots & \vdots & \ddots & \vdots \\
        -1 & -1 & \cdots & 4n-1\\
	\end{pmatrix}\] 
	and 
        \[ B =\displaystyle \begin{pmatrix}
		4n+3 & -1 & \cdots & -1\\
        -1 & 4n+3 & \cdots & -1\\
        \vdots & \vdots & \ddots & \vdots\\
        -1 & -1 & \cdots & 4n+3\\
	\end{pmatrix},\]\\ where orders of $A$ and $B$ are $4n-4\times 4n-4$ and $4n\times 4n$, respectively. $\mathcal{O}$ is zero matrix of order $4n-4 \times \ 4n$ and $\mathcal{O'}$ is the transpose of the $\mathcal{O}$ matrix. It follows that, $A= 4nI_{4n-4}-J_{4n-4}$ and $B=(4n+4)I_{4n} - J_{4n}.$        
\\
Then the characteristic polynomial of $L({CSCom}(SD_{8n}))$ is
	\[\Phi(L(CSCom(SD_{8n})), x)
 =\displaystyle
 \begin{vmatrix}
	x-(8n-1) & 1 & 1 & 1 & \cdots & \cdots & 1   \\
	1 &x-(8n-1) & 1 & 1  & \cdots & \cdots & 1 \\
        1 & 1 & x-(8n-1) & 1 & \cdots & \cdots & 1   \\
	1 & 1 & 1 & x-(8n-1)  & \cdots & \cdots & 1  \\
	1  & 1&1 & 1 & xI_{4n-4}-A &  &  \mathcal{O}  \\
 
         \vdots & \vdots & \vdots & \vdots &  \\
		1  & 1 &1&1&\mathcal{O'}& & xI_{4n}-B  \\
\end{vmatrix}.
\]
After applying the following row operations consecutively
	\begin{itemize}
		\item $R_1 \rightarrow (x -1)R_1 - R_2  - \cdots - R_{8n}.$
		\item $R_2 \rightarrow (x -2)R_2 - R_3 - \cdots - R_{8n}.$
        \item $R_3 \rightarrow (x -3)R_3 - R_4 - \cdots - 
        R_{8n}.$
        \item $R_4 \rightarrow (x -4)R_4 - R_5 - \cdots - 
        R_{8n}.$
		
	\end{itemize}	
	
 we obtain, 
	
	\[\Phi(L(CSCom(SD_{8n})), x) = \frac{x(x -8n)^4}{(x-4)}  \displaystyle \begin{vmatrix}
	xI_{4n-4} -A & \mathcal O\\
	\mathcal O' & xI_{4n} - B
	\end{vmatrix} = \frac{x(x -8n)^4}{(x-4)} |xI_{4n-4}- A|\cdot|xI_{4n} - B|.\]
	
\[\Phi(L({CSCom}(SD_{8n})), x), x) = \frac{x(x -8n)^4}{(x-1)} |xI_{4n-4} - A|\cdot |xI_{4n}-B|.\tag{3}\] Clearly, 
	$|xI_{4n-4}-A| =(x -4)(x-4n)^{4n-5}$ 
 and	$|xI_{4n}-B| =(x -4)(x-(4n+4))^{4n-1}$. Thus, we obtain the proposed characteristic polynomial.
\end{proof}
\begin{corollary}
If $n$ is odd, then the Laplacian spectrum of $CSCom(SD_{8n})$ is given by 
\[\displaystyle \begin{pmatrix}
0 & 4 & 4n &  4n+4 & 8n\\
 1 & 1 & 4n-5 & 4n-1 & 4\\
\end{pmatrix}.\]
\end{corollary}

By {\cite[Corollary 4.2]{mohar1991spectrum}}, we have the following corollary. 

\begin{corollary}
If $n$ is odd, then the number of spanning trees of $CSCom(SD_{8n})$ are $ 2^{8n+1} n^{4n-2} (4n+4)^{4n-1}.$
\end{corollary}
\begin{theorem}\label{CSCom-even-SD_8n}  
If $n$ is even, then the characteristic polynomial of the Laplacian matrix of ${CSCom}(SD_{8n})$ is given by
\[\Phi(L({CSCom}(SD_{8n})), x) = x(x-2)^2(x-(2n+2))^{4n-2}(x -4n)^{4n-3}(x-8n)^{2}.\]
\end{theorem}
\begin{proof}
The Laplacian matrix $L({CSCom}(SD_{8n}))$ is an $8n \times 8n$ matrix given below, in which the rows and columns are arranged in order by the vertices $e = a^{4n}=b^{2},a^{2n}, a,a^2, \ldots, a^{2n-1},a^{2n+1},\ldots ,a^{4n-1}$ followed by $ab, a^3b,a^5b, \ldots,a^{4n-1}b$ and finally $b,a^2b,a^4b,\ldots, a^{4n-2}b$.

\[L(CSCom(SD_{8n}))  = \displaystyle \begin{pmatrix}
 8n-1 & -1 & -1 & \cdots & -1 & -1 & \cdots & -1 & -1 & \cdots~ -1 \\
-1 & 8n-1 & -1 &  \cdots & -1 & -1 & \cdots & -1 & -1 & \cdots ~ -1 \\
-1  & -1 & -1 & ~ A & & & \mathcal{O} & & & \mathcal{O }\\
 \vdots & \vdots & & & & & & \\ 
	-1  & -1 & -1  & ~\mathcal{O}' & & & B & &  &\mathcal{O}"      \\ 
  \vdots & \vdots& & &  &  &  & \\ 
	-1  & -1 & -1 & ~ \mathcal{O}' & & &\mathcal{O}" & & & B  \\
	\end{pmatrix},\]
  where
	\[A =\displaystyle \begin{pmatrix}
	4n-1 & -1 & \cdots & -1\\
        -1 & 4n-1 & \cdots & -1\\
        \vdots & \vdots & \ddots & \vdots\\
        -1 & -1 & \cdots & 4n-1\\
	\end{pmatrix}\] 
and	 
        \[ B =\displaystyle \begin{pmatrix}
		2n+1 & -1 & \cdots & -1\\
        -1 & 2n+1 & \cdots & -1\\
        \vdots & \vdots & \ddots & \vdots\\
        -1 & -1 & \cdots & 2n+1\\
	\end{pmatrix},
 \]  \\ such that the orders of $A$ and $B$ are $4n-2\times 4n-2$ and $2n\times 2n$, respectively. It follows that $A= 4nI_{4n-2}-J_{4n-2}$ and $B= (2n+2)I_{2n}-J_{2n}$.
 
Then the characteristic polynomial of $L({CSCom}(SD_{8n}))$ is\\
 \[\Phi(L(CSCom(SD_{8n})), x)  = \displaystyle \begin{vmatrix}
	x-(8n-1) & 1 &  \cdots & \cdots & \cdots & \cdots & 1  \\
	1 &x-(8n-1) & \cdots &  \cdots & \cdots & \cdots & 1 \\	
    1 & 1 & xI_{4n-2}-A &  &\mathcal{O}  &  & \mathcal{O} & \\
    1 & 1 & & & & & & \\
    \; \; \vdots & \; \;\vdots&    &  &  &  &  & \\
    1 & 1 & \mathcal{O'} &  & xI_{2n}-B  &   &  \mathcal{O"} \\
    \; \; \vdots & \; \;\vdots&    &  &  &  &  & \\
  1 & 1 & & & & & & \\ 
      1 & 1 & \mathcal{O'} &  & \mathcal{O"}  &   &  xI_{2n}-B \\ 
	\end{vmatrix}.
 \]
\\
 After applying the following row operations consecutively,
	\begin{itemize}
		\item $R_1 \rightarrow (x -1)R_1 - R_2  - \cdots - R_{8n}.$
		\item $R_2 \rightarrow (x -2)R_2 - R_3 - \cdots - R_{8n}.$
	\end{itemize}	
 We obtain, 
	
	\[\Phi(L(CSCom(SD_{8n})), x) = \frac{x(x -8n)^2}{(x-2)}  \displaystyle \begin{vmatrix}
	xI_{4n-2} -A & \mathcal O & \mathcal O\\
	\mathcal O' & xI_{2n} - B & \mathcal{O"}\\
 \mathcal{O'} & \mathcal{O"}  & xI_{2n} - B
	\end{vmatrix} \]
 \vspace{3mm}
 \[
 = \frac{x(x -8n)^2}{(x-2)} |xI_{4n-2}- A|\cdot|xI_{2n} - B|^2.\]
It follows that the characteristic polynomial of ${CSCom}(SD_{8n})$, when $n$ is even, is
$$x(x-2)^2(x-(2n+2))^{4n-2}(x -4n)^{4n-3}(x-8n)^{2}.$$

\end{proof}
\begin{corollary}
If $n$ is even, then the Laplacian spectrum of $CSCom(SD_{8n})$ is given by 
\[\displaystyle \begin{pmatrix}
0 & 2 & 2n+2 &  4n & 8n\\
 1 & 2 & 4n-2 & 4n-3 & 2\\
\end{pmatrix}.\]
\end{corollary}

By {\cite[Corollary 4.2]{mohar1991spectrum}}, we have the following corollary.

\begin{corollary}
If $n$ is even, then the number of spanning trees of $CSCom(SD_{8n})$ are $ 2^{8n-1} n^{4n-2}(2n+2)^{4n-2}.$
\end{corollary}
\vspace{5pt}

\vspace{1cm}
\noindent
{\bf Varun J Kaushik\textsuperscript{\normalfont 1}, \bf Ekta\textsuperscript{\normalfont 1}, \bf Parveen\textsuperscript{\normalfont 2}, Jitender Kumar\textsuperscript{\normalfont 1}}
\bigskip

\noindent{\bf Addresses}:

\end{document}